\documentclass{article}
\usepackage[english]{babel}

\usepackage[letterpaper,top=2cm,bottom=2cm,left=2.7cm,right=2.7cm,marginparwidth=1.75cm]{geometry}

\usepackage{amsmath, amsthm, amssymb}
\usepackage{mathtools}   
\usepackage{graphicx}
\usepackage[colorlinks=true, allcolors=blue]{hyperref}
\usepackage{lastpage}			
\usepackage{indentfirst}		
\usepackage{xurl}
\usepackage{calrsfs}
\usepackage{setspace}
\usepackage[shortlabels]{enumitem}
\usepackage{comment}
\usepackage{titlesec}
\titleformat{\section}[block]{\large\scshape\centering}{\thesection.}{1em}{}

\newcommand{\pmat}[1]{\begin{pmatrix}#1\end{pmatrix}}
\newcommand{\T}{\Tilde}

\newcommand{\G}[1]{\left\langle#1\right\rangle}

\newcommand{\codim}{\mathrm{codim}}
\newcommand{\Span}{\mathrm{span}}

\newtheorem{theorem}{Theorem}[section]
\newtheorem{corollary}{Corollary}[section]
\newtheorem{lemma}{Lemma}[section]

\theoremstyle{definition}
\newtheorem{definition}{Definition}[section]
\theoremstyle{definition}
\newtheorem{example}{Example}[section]
\theoremstyle{definition}
\newtheorem{remark}{Remark}[section]
\theoremstyle{definition}
\newtheorem{open}{Open Problem}

\newcommand{\Addresses}{{
  \bigskip

  P.~Emerick (Corresponding author), \textsc{Departamento de Matemática, Universidade Federal de Juiz de Fora,
    Juiz de Fora, Brazil}\par\nopagebreak
  \textit{E-mail address}
  \texttt{pedro.emerick@estudante.ufjf.br}

  \medskip

  L.A.~Belmonte , \textsc{Departamento de Matemática,
    Universidade de São Paulo Campus São Carlos, São Carlos, Brazil}\par\nopagebreak
  \textit{E-mail address}
  \texttt{luanarjuna@usp.br}

}}

\providecommand{\keywords}[1]
{	
  \textbf{\small{Keywords}} #1
}

\providecommand{\MSC}[1]
{	
  \textbf{\small{2020 MSC}} #1
}

\title{Infinite Lineability: On the Abundance of Dense Subspaces}

\date{}

\author{Pedro Emerick\textsuperscript{1}, Luan Arjuna Belmonte\textsuperscript{1}}

\begin{document}
    
\maketitle

\begin{abstract}
   In this paper, we investigate the concept of infinite dense-lineability recently introduced by M. Calderón-Moreno, P. Gerlach-Mena and J. Prado-Bassas. We answer a question posed by the authors about the equivalence between infinite (pointwise) dense-lineability and (pointwise) dense-lineability. We prove that the equivalence always holds in first-countable topological vector spaces and under some assumptions about the weight of the topology. However, the equivalence is not always true, as shown in an example. Furthermore, we introduce the notions of infinite $(\alpha,\beta)$-dense-lineability and infinite (strongly) dense-algebrability and obtain some analogous results in these cases. We also obtain a criterion for strongly dense-algebrability for sets of the form $X\setminus Y$, where $X$ is a free algebra and $Y$ is a free subalgebra of $X$.

\end{abstract}


\MSC{46B87, 15A03, 46A16, 46J30}
    
\footnotetext[1]{The author was financed in part by the Coordenação de Aperfeiçoamento de Pessoal de Nível Superior – Brazil (CAPES) – Finance Code 001}

\section{Introduction}\label{Int}

The study of lineability focuses on the search for large linear structures inside subsets of vector spaces that enjoy some pathological or special properties. This branch of research was initiated by V. Gurariy \cite{Gurariy66}, who proved that the set of nowhere differentiable functions on the interval $[0, 1]$ contains an infinite-dimensional subspace. After \cite{GURARIY2004,Aron2005}, lineability problems attracted much attention from researchers worldwide. Several variations of lineability were introduced, such as dense-lineability and spaceability (where we seek large dense linear structures and large closed linear structures, respectively). Another related research field is algebrability, i.e., the search for large subalgebras inside subsets of associative linear algebras.

To refer the reader to some articles on lineability, we have \cite{Botelho2011,Botelho2012,ARON2018186, Papathanasiou, FAVARO2020144}. General criteria for lineability are scarcer, but often very useful (see \cite{BERNALGONZALEZ20143997,araujo2024general}). Research on algebrability is plentiful and may be found, for instance, on \cite{Papathanasiou,Glab2013,Bes2020}. A wider exposure of lineability, algebrability and related subjects can be found in \cite{aron2015lineability}.

In more precise terms, given $X$ a vector space, $M$ a subset of $X$ and $\alpha$ a cardinal number, we say that $M$ is $\alpha$-lineable when there is an $\alpha$-dimensional subspace $Y$ such that $Y\subset M\cup\{0\}$.

Let $X$ be a topological vector space (TVS), that is, a vector space endowed with a topology for which its sum and scalar multiplication are continuous. We say that $M$ is $\alpha$-dense-lineable when there is an $\alpha$-dimensional dense subspace $Y$ such that $Y\subset M\cup\{0\}$. When $\alpha=\aleph_0$, we usually omit the cardinal number and say only that $M$ is lineable or dense-lineable.

With the development of the theory, it became clear that positive responses to lineability problems are quite common, even in very exotic sets. This made more restrictive notions necessary. The $(\alpha,\beta)$-lineability was introduced by V. Fávaro, D. Pellegrino and D. Tomaz (see \cite{Fávaro_a_b_lineability}), and the pointwise-lineability by D. Pellegrino and A. Raposo Jr (see \cite{Pellegrino_pointwise}). We define those notions below.

If $X$ is a vector space, $M$ is a subset of $X$ and $\alpha,\beta$ are cardinal numbers, we say that
\begin{itemize}
    \item $M$ is $(\alpha,\beta)$-lineable when $M$ is $\alpha$-lineable and, for each $\alpha$-dimensional subspace $W$ with $W\subset M\cup\{0\}$, there is a $\beta$-dimensional subspace $Y$ such that $W\subset Y\subset M\cup\{0\}$.
    \item $M$ is pointwise $\alpha$-lineable when, for each $x\in M$ there is an $\alpha$-dimensional subspace $Y$ such that $x\in Y\subset M\cup\{0\}$.
\end{itemize}

When $X$ is endowed with a topology, the concepts of $(\alpha,\beta)$-dense-lineability and pointwise $\alpha$-dense-lineability are defined naturally, i.e., requiring that $Y$ can be taken as a dense subspace.

Until recently published paper \cite{Calderon-Moreno2023} from M.  Calderón-Moreno, P. Gerlach-Mena and J. Prado-Bassas, the abundance of linear structures contained in lineable subsets was an overlooked question. To investigate that subject, the authors introduced the definitions of infinite pointwise lineability and infinite pointwise dense-lineability. They are established as follows.

If $X$ is a vector space, $M$ is a subset of $X$ and $\alpha\geq\aleph_0$ is a cardinal number, we say that $M$ is infinitely pointwise $\alpha$-lineable if, for every $x\in M$, there exists a family $\{Y_k\}_{k\in\mathbb N}$ of vector subspaces such that for each $k\in\mathbb N$, we have $\dim(Y_k)=\alpha$, $x\in Y_k\subset M\cup\{0\}$ and $Y_k\cap Y_l=\Span(x)$ for any $l\in\mathbb N$, $l\neq k$. When $X$ is endowed with a topology and each $Y_k$ is dense in $X$, we say that $M$ is infinitely pointwise $\alpha$-dense-lineable.

The following question was stated by the authors (\cite[Open Problem 1]{Calderon-Moreno2023}): Let $X$ be a topological vector space and $M\subset X$ be a (pointwise) $\alpha$-dense-lineable set. Is $M$ always infinitely (pointwise) $\alpha$-dense-lineable?

In this paper, we expand the concepts analyzed, introducing the idea of $\alpha$-infinite (pointwise) $\beta$-dense-lineability, that is, we require not only infinitely many $\beta$-dimensional dense subspaces in $M\cup\{0\}$, but also that there is at least $\alpha$, where $\alpha\geq\aleph_0$. We also introduce analogous definitions for infinite $(\alpha,\beta)$-dense-lineability.

Our main goal is to respond to the aforementioned problem by providing a positive partial answer for first-countable TVS (we prove the stronger equivalence between $\alpha$-infinite $\alpha$-dense-lineability and $\alpha$-dense-lineability and an analogous result for the pointwise case) and providing a counterexample where the equivalence fails. Furthermore, we prove that the analogous question formulated for infinite $(\alpha,\beta)$-dense-lineability has the same answer.

This paper is organized as follows. In Section \ref{Div TVS}, we provide a criterion for a TVS to have infinitely many dense subspaces that only intersect two-by-two at $\{0\}$ based on the relation between the weight of the topology and the dimension of the TVS. This result will be a key tool in Section \ref{inf lineability}. We also give an example of a TVS without any proper dense subspace. Section \ref{inf lineability} addresses conditions for the equivalence between infinite dense-lineability and dense-lineability as mentioned above.

In Section \ref{alg}, we formally define algebrability, introduce the concept of infinite dense-algebrability and prove that there is an equivalence between strongly dense-algebrability and infinite strongly dense-algebrability when the algebra is first-countable. We also provide an example showing that infinite dense-algebrability is not always equivalent to dense-algebrability. As part of the efforts to prove the main results of the section, we obtain a general criterion for strongly dense-algebrability in sets of the form $X\setminus Y$, where $X$ is a commutative topological free algebra and $Y$ is a free subalgebra of $X$.

\section{Division of a topological vector space in dense subspaces}\label{Div TVS}

The fact that infinite pointwise $\alpha$-lineability is equivalent to pointwise $\alpha$-lineability, when $\alpha\geq\aleph_0$, was proven in \cite[Proposition 2.1]{Calderon-Moreno2023}. The key to this result is the possibility of dividing an infinite-dimensional vector space into infinitely many linearly independent subspaces with maximal dimension. Therefore, if we can split a TVS as before, but with dense subspaces, we can repeat a similar argument for infinite pointwise $\alpha$-dense-lineability.

Before proceeding, we recall that the weight of a topological space $X$, denoted by $w(X)$, is the smallest cardinality of a base for its topology. Additionally, throughout this work, the vector spaces are taken over $\mathbb K$ where $\mathbb K=\mathbb R$ or $\mathbb K=\mathbb C$. Furthermore, the cardinality of a set $M$ is denoted by $|M|$ and we use the notation $\aleph_0:=|\mathbb N|$.

As we prove in Theorem \ref{many dense subspaces}, if $X$ is a TVS and $w(X)\le\dim(X)=\alpha$, then $X$ can be divided into $\alpha$ linearly independent subspaces, all dense and $\alpha$-dimensional. An example of an infinite-dimensional TVS without any proper dense subspace (Example \ref{no dense subspace}) is also presented.

We start with an adaptation of \cite[Lemma 3.1]{Leonetti_weight_dense_lineability}, which will play an important role in the proof of Theorem \ref{many dense subspaces}.

\begin{lemma}\label{X-Y dense lineable}
    Let $X$ be a topological vector space and $Y$ a linear subspace such that $w(X)\le\codim(Y)=\alpha$. Then $X\setminus Y$ is $\alpha$-dense-lineable in $X$.
\end{lemma}
\begin{proof}
    Let $\{U_\kappa\}_{\kappa<\alpha}$ be a topological basis for $X$ (we do not require index uniqueness). Assume that every $U_\kappa$ is non-empty. We build by transfinite induction elements $\{x_\kappa\}_{\kappa<\alpha}$ such that
    $$x_\kappa\in B_\kappa\setminus\Span(Y\cup\{x_\lambda\}_{\lambda<\kappa}).$$
    As $Y\subsetneq X$, $\mathrm{int}(Y)=\varnothing$ and there is $x_0\in B_0\setminus Y$. Let $\kappa<\alpha$ and suppose, by transfinite induction hypothesis, that $\{x_\lambda\}_{\lambda<\kappa}$ is already defined. Let $Y_\kappa:=\Span(Y\cup\{x_\lambda\}_{\lambda<\kappa})$. Note that $\codim(Y_\kappa)=\codim(Y)=\alpha$, so $Y_\kappa\subsetneq X$ and $\mathrm{int}(Y)=\varnothing$, hence, there is $x_\kappa\in B_\kappa\setminus Y_\kappa$. This shows the existence of the vectors $\{x_\kappa\}_{\kappa<\alpha}$. The subset $\{x_\kappa\}_{\kappa<\alpha}$ is linearly independent and dense in $X$, therefore, $Z:=\Span(\{x_\kappa\}_{\kappa<\alpha})\subset (X\setminus Y)\cup\{0\}$ is $\alpha$-dimensional and dense in $X$; thus, $X\setminus Y$ is $\alpha$-dense-lineable in $X$.
    
\end{proof}

\begin{theorem}\label{many dense subspaces}
    Let $X$ be a topological vector space such that $w(X)\leq\dim(X)=\alpha$. Then, there is a family $\{Y_\kappa\}_{\kappa<\alpha}$ such that:
    \begin{enumerate}
        \item $Y_\kappa$ is an $\alpha$-dimensional dense linear subspace of $X$ for all $\kappa<\alpha$;
        \item $\{Y_\kappa\}_{\kappa<\alpha}$ is linearly independent (in particular, $Y_{\kappa_1}\cap Y_{\kappa_2}=\{0\}$ when $\kappa_1,\kappa_2<\alpha$ and $\kappa_1\neq\kappa_2$).
    \end{enumerate}
\end{theorem}
\begin{proof}
    Let $X = \bigoplus_{\kappa<\alpha} X_\kappa$ with $\dim(X_\kappa)=\alpha$ for each $\kappa<\alpha$. We build by transfinite induction a family $\{Y_\kappa\}_{\kappa<\alpha}$ of $\alpha$-dimensional dense linear subspaces of $X$ such that for each $\kappa<\alpha$ the following sum is direct
    \begin{equation}\label{HI}
        \left(\bigoplus_{\mu>\kappa} X_\mu\right)\oplus\left(\bigoplus_{\mu\le\kappa} Y_\mu\right).
    \end{equation}
    
    Let $\kappa<\alpha$ and suppose, by transfinite induction hypothesis, that $\{Y_\lambda\}_{\lambda<\kappa}$ is already defined such that for each $\lambda<\kappa$ the sum $(\bigoplus_{\mu>\lambda} X_\mu)\oplus(\bigoplus_{\mu\le\lambda} Y_\mu)$ is direct. Then $S := (\bigoplus_{\mu\ge\kappa} X_\mu)\oplus(\bigoplus_{\mu<\kappa} Y_\mu)$ is a direct sum. In fact, let $s\in S$ be so that
    $$\begin{aligned}
        s
        &= x_1+x_2+\dots+x_n+y_1+y_2+\dots+y_m\\
        &= \T x_1+\T x_2+\dots+\T x_n+\T y_1+\T y_2+\dots+\T y_m
    \end{aligned}$$
    where $x_i,~\T x_i\in X_{\zeta_i}$ and $y_j,~\T y_j\in Y_{\eta_j}$ with $\zeta_i \ge \kappa > \eta_j$, $i\in\{1, \dots, n\}$, $j\in\{1, \dots, m\}$. Let $\lambda$ be so that $\zeta_i\ge\kappa>\lambda\ge\eta_j$ for all $ i\in\{1, \dots, n\}$ and $j\in\{1, \dots, m\}$. Because the sum $(\bigoplus_{\mu>\lambda} X_\mu)\oplus(\bigoplus_{\mu\le\lambda} Y_\mu)$ is direct, then $x_i = \T x_i$ and $y_j = \T y_j$ for all $i\in\{1, \dots, n\}$ and $j\in\{1, \dots, m\}$, as we wanted.
    
    Then, we have $\codim((\bigoplus_{\lambda>\kappa}X_\lambda)\oplus(\bigoplus_{\lambda<\kappa} Y_\lambda))\ge \dim(X_\kappa)=\alpha\ge w(X)$, so, by Lemma \ref{X-Y dense lineable}, there is an $\alpha$-dimensional dense subspace $Y_\kappa$ such that $Y_\kappa\cap((\bigoplus_{\lambda>\kappa} X_\lambda)\oplus(\bigoplus_{\lambda<\kappa} Y_\lambda)) = \{0\}$, therefore
    $$Y_\kappa\oplus\left(\left(\bigoplus_{\lambda>\kappa} X_\lambda\right)\oplus\left(\bigoplus_{\lambda<\kappa} Y_\lambda\right)\right) = \left(\bigoplus_{\lambda>\kappa} X_\lambda\right)\oplus\left(\bigoplus_{\lambda\le\kappa} Y_\lambda\right)$$
    is a direct sum. This shows the existence of the family $\{Y_\kappa\}_{\kappa<\alpha}$.
\end{proof}

The next example shows that we cannot remove the hypothesis $w(X)\leq\dim(X)$.

\begin{example}\label{no dense subspace}
    There are TVSs with no proper dense subspaces.

    For each $n\in\mathbb N$, let $u_n:\mathbb R^n\rightarrow \mathbb R^{\mathbb N}$ denote the inclusion map $u_n(x_1,\dots,x_n)=(x_1,\dots,x_n,0,0,\dots)$. In each $u_n(\mathbb R^n)$ we consider the topology induced by $u_n$. Let $\mathbb R^\infty := \bigcup_{n=1}^\infty u_n(\mathbb R^n)$ be the set of all sequences with finite support endowed with the inductive limit topology. For more details on that topology, we recommend \cite[Chapter 12]{beckenstein2010} and \cite[Chapter 5]{robertson1973}.
    
    Then $\mathbb R^\infty$ is a TVS and $S\subseteq\mathbb R^\infty$ is open (resp. closed) if and only if for every $n\in\mathbb N$, $S\cap u_n(\mathbb R^n)$ is open (resp. closed) in $u_n(\mathbb R^n)$ \cite[Proposition 5.4]{robertson1973}.

    Let us show that $w(\mathbb R^\infty)>\aleph_0 = \dim(\mathbb R^\infty)$. Let $\{e_n\}_{n\in\mathbb N}$ be the canonical basis for $\mathbb R^\infty$ and $\{U_n\}_{n\in\mathbb N}$ be a countable family of open neighborhoods of $0$. Since $\mathbb R^\infty$ is a TVS, for each $n\in\mathbb N$ there is a non-zero $r_n\in\mathbb R$ such that $r_ne_n\in U_n$. Let $F:=\{r_ne_n:n\in\mathbb N\}$. Then, for each $n\in\mathbb N$, $F\cap u_n(\mathbb R^n)$ is finite (hence closed), therefore, $F$ is closed in $\mathbb R^\infty$ and $U:=\mathbb R^\infty\setminus F$ is an open neighborhood of $0$ that does not contain any $U_n$. This shows that there is no countable neighborhood basis for $0$, so $\mathbb R^\infty$ is not even first-countable.

    Now, we verify that there is no proper dense subspace of $\mathbb R^\infty$. Let $W$ be a subspace of $\mathbb R^\infty$. Then, for each $n\in\mathbb N$, $\dim(W\cap u_n(\mathbb R^n))\le \dim(u_n(\mathbb R^n))=n<\infty$. Therefore, for each $n\in\mathbb N$, $W\cap u_n(\mathbb R^n)$ is closed, i.e., $W$ is closed in $\mathbb R^\infty$.
    We conclude that $W$ is dense if and only if $W=\mathbb R^\infty$.
\end{example}

Motivated by the example above, one could conjecture that an infinite-dimensional TVS has infinitely many proper dense subspaces if and only if $w(X)\le \dim(X)$. Let us show that this is not the case.

\begin{example}\label{d<w with dense subspaces}
Let $V$ be an $\aleph_0$-dimensional vector space endowed with the trivial topology. Then $V$ is a TVS.
In this setting, the vector space $X:=\mathbb R^\infty\times V$ endowed with the product topology is a TVS \cite[Example 4.7.1]{beckenstein2010}.
Because $\mathbb R^\infty$ is not first-countable, the same holds for $X$. On the other hand, since $\dim(V)=\dim(\mathbb R^\infty)=\aleph_0$, we have $\dim(X)=\aleph_0$, so $\dim(X)<w(X)$.

Let us show now that $X$ has infinitely many dense subspaces whose pairwise intersection is trivial. Let $\mathcal B:=\{e_m:m\in\mathbb N\}$ and $\mathcal C:=\{f_m^n: m,n\in\mathbb N\}$ be a basis for $\mathbb R^\infty$ and $V$, respectively. Let $X_n:=\Span(\{(e_m, f_m^n):m\in\mathbb N\})$. Then $X_i\cap X_j=\{0\}$ since $\Span(\mathcal C_i)\cap\Span(\mathcal C_j)=\{0\}$, $i\ne j$. Moreover, each $X_n$ is dense in $X$ since $\{U\times V: U\subset\mathbb R^\infty\textnormal{ open}\}$ is a topological basis for $X$ and, given $x\in U\subset \mathbb R^\infty$ open, there is $v\in V$ such that $(x, v)\in X_n$, therefore $(U\times V)\cap X_n\ne\varnothing$.

\end{example}

\section{Applications to infinite dense-lineability}\label{inf lineability}

We are interested in analyzing some new concepts of infinite lineability in addition to those presented in \cite{Calderon-Moreno2023}; therefore, we propose the following definitions.

\begin{definition}\label{def infinite lineability}
    Given a topological vector space $X$, a subset $M$ of $X$ and cardinal numbers $\alpha,\beta,\gamma$, with $\alpha\ge\aleph_0$ and $\beta>\gamma$, we say that $M$ is
    \begin{itemize}
        \item $\alpha$-infinitely $\beta$-dense-lineable when there is a family $\{Y_\kappa\}_{\kappa<\alpha}$ of $\beta$-dimensional dense subspaces of $X$ with $Y_\kappa\subset M\cup\{0\}$ for each $\kappa<\alpha$ and $Y_{\kappa_1}\cap Y_{\kappa_2}=\{0\}$ when $\kappa_1\neq\kappa_2$.
        \item $\alpha$-infinitely pointwise $\beta$-dense-lineable when for each $x\in M$ there is a family $\{Y_\kappa\}_{\kappa<\alpha}$ of $\beta$-dimensional dense subspaces of $X$ with $x\in Y_\kappa\subset M\cup\{0\}$ for each $\kappa<\alpha$ and $Y_{\kappa_1}\cap Y_{\kappa_2}=\Span(x)$ when $\kappa_1\neq\kappa_2$.
        \item $\alpha$-infinitely $(\gamma,\beta)$-dense-lineable when $M$ is $\gamma$-lineable and, for each $\gamma$-dimensional subspace $W$ contained in $M\cup\{0\}$, there is a family $\{Y_\kappa\}_{\kappa<\alpha}$ of $\beta$-dimensional dense subspaces of $X$ with $W\subset Y_\kappa\subset M\cup\{0\}$ for each $\kappa<\alpha$ and $Y_{\kappa_1}\cap Y_{\kappa_2}=W$ when $\kappa_1\neq\kappa_2$.
    \end{itemize}
\end{definition}

In Theorem \ref{equivalences1} and Theorem \ref{equivalences2} we provide conditions for the equivalence between $\alpha$-dense-lineability (resp. pointwise $\alpha$-dense-lineability, $(\gamma,\alpha)$-dense-lineability) and $\alpha$-infinite $\alpha$-dense-lineability (resp. $\alpha$-infinite pointwise $\alpha$-dense-lineability, $\alpha$-infinite $(\gamma,\alpha)$-dense-lineability).

In Corollary \ref{equivalence for pseudometrizable}, we prove that the answer for \cite[Open Problem 1]{Calderon-Moreno2023} is positive for every first-countable TVS and that an equivalent result for infinite $(\gamma,\alpha)$-dense-lineability also holds.

Example \ref{counterexample} proves that the answer for \cite[Open Problem 1]{Calderon-Moreno2023} is false in general.

\begin{theorem}\label{equivalences1}
    Let $X$ be a topological vector space, $M$ a subset of $X$ and $w(X)\leq\alpha$. Then $M$ is $\alpha$-dense-lineable if, and only if, $M$ is $\alpha$-infinitely $\alpha$-dense-lineable.
\end{theorem}
\begin{proof}
    Let $Y\subset M\cup\{0\}$ be a dense subspace of $X$ with $\dim(Y)=\alpha$. Since $w(Y)\leq w(X)\leq\alpha$, by Theorem $\ref{many dense subspaces}$, there is a family $\{Y_{\kappa}\}_{\kappa<\alpha}$ of dense linear subspaces of $Y$ with $\dim(Y_\kappa)=\alpha$ for each $\kappa<\alpha$ and $Y_{\kappa_1}\cap Y_{\kappa_2}=\{0\}$ for $\kappa_1\neq\kappa_2$. Because $Y$ is dense in $X$, each $Y_{\kappa}$ is dense in $X$. Hence, $\{Y_\kappa\}_{\kappa<\alpha}$ is a family of $\alpha$-dimensional dense subspaces of $X$ contained in $M\cup\{0\}$, i.e., $M$ is $\alpha$-infinitely $\alpha$-dense-lineable. 
\end{proof}

\begin{theorem}\label{equivalences2}
    Let $X$ be a topological vector space, $M$ a subset of $X$ and $w(X)\leq\alpha$ where $\alpha\geq\aleph_0$. Then
    \begin{enumerate}[label=(\roman*)]
        \item $M$ is pointwise $\alpha$-dense-lineable if, and only if, $M$ is $\alpha$-infinitely pointwise $\alpha$-dense-lineable.
        \item $M$ is $(\gamma,\alpha)$-dense-lineable if, and only if, $M$ is $\alpha$-infinitely $(\gamma,\alpha)$-dense-lineable.
    \end{enumerate}
\end{theorem}
\begin{proof}
    Assume $M$ is $(\gamma, \alpha)$-dense-lineable. Then, given a $\gamma$-dimensional subspace $W\subset M\cup\{0\}$, there is an $\alpha$-dimensional dense subspace $Y\subset M\cup\{0\}$ such that $W\subset Y$. Because $\dim(Y)=\alpha>\dim(W)$, we have $\codim_Y(W)=\alpha\ge w(X)\ge w(Y)$; therefore, by Lemma \ref{X-Y dense lineable}, there is an $\alpha$-dimensional dense subspace $Z\subset (Y\setminus W)\cup\{0\}$. Since $w(Z)\le w(Y)\le\alpha$, we can apply Theorem \ref{many dense subspaces} to obtain a family $\{Z_\kappa\}_{\kappa<\alpha}$ of $\alpha$-dimensional linearly independent dense subspaces of $Z$.
    
    For $\kappa<\alpha$, let $Y_\kappa:=Z_\kappa\oplus W\subset Y$. Of course, each $Y_\kappa$ is dense in $Y$; therefore, in $X$. Let $\kappa_1, \kappa_2<\alpha$, $\kappa_1\ne\kappa_2$, and $y\in Y_{\kappa_1}\cap Y_{\kappa_2}$. Then $y=z_1+w_1=z_2+w_2$ with $z_1\in Z_{\kappa_1}$, $z_2\in Z_{\kappa_2}$ and $w_1, w_2\in W$. However, $Z_{\kappa_1}$, $Z_{\kappa_2}$ and $W$ are linearly independent, therefore $z_1=z_2=0$ and $y=w_1=w_2\in W$. This shows that $Y_{\kappa_1}\cap Y_{\kappa_2} = W$; hence, $M$ is $\alpha$-infinitely $(\gamma, \alpha)$-dense-lineable.

    Now, assume $M$ is pointwise $\alpha$-dense-lineable. In particular, $M$ is $(1, \alpha)$-dense-lineable; therefore, $\alpha$-infinitely $(1, \alpha)$-dense-lineable, i.e., given $x\in M$ such that $\Span(x)\subset M\cup\{0\}$, there is a family $\{Y_\kappa\}_{\kappa<\alpha}$ of $\alpha$-dimensional dense subspaces of $X$ such that $Y_\kappa\subset M\cup\{0\}$ for each $\kappa<\alpha$ and $Y_{\kappa_1}\cap Y_{\kappa_2}=\Span(x)$. Since $\Span(x)\subset M\cup\{0\}$ for all $x\in M$, we have $M$ $\alpha$-infinitely pointwise $\alpha$-dense-lineable.

    The reverse implications are trivial.

\end{proof}

\begin{example}\label{counterexample}
    
    Consider the setting from Example \ref{d<w with dense subspaces}.
    Notice that $M=\mathbb R^\infty\times\{0\}$ is a dense subspace of $X$. Since $\mathbb R^\infty$ does not admit any proper dense subspace, $M$ is dense-lineable but has only one dense subspace inside.
    
    Furthermore, there are subsets of TVSs that are $n$-dense-lineable for all $n\in\mathbb N$, but not infinitely dense-lineable.

    Notice that for each $a\in\mathbb R^\infty$, there exists exactly one $b\in V$ such that $(a, b)\in X_n$. In fact, if $\lambda_1, \lambda_2, \dots, \lambda_k$ are scalars such that
    $$\begin{aligned}
    \pmat{a\\b}
    =\lambda_1\pmat{e_1\\f_1^n}+\lambda_2\pmat{e_2\\f_2^n}+\dots+\lambda_k\pmat{e_k\\f_k^n}
    =\pmat{\lambda_1 e_1+\lambda_2 e_2+\dots+\lambda_k e_k\\\lambda_1 f_1^n+\lambda_2 f_2^n+\dots+\lambda_k f_k^n}
    \end{aligned}$$
    then, since $\{e_m: m\in\mathbb N\}$ is a basis for $\mathbb R^\infty$, $(\lambda_1, \lambda_2, \dots, \lambda_k)$ is uniquely determined by $a$, therefore, so is $b$.

    This shows the projection map $\pi:X_n\subset\mathbb R^\infty\times V\to\mathbb R^\infty$ is bijective, therefore, $X_n$ is isomorphic to $\mathbb R^\infty$. Indeed, any projection is continuous and, since the open sets of $X_n$ have the form $(U\times V)\cap X_n$ with $U$ an open set of $\mathbb R^\infty$ and $\pi((U\times V)\cap X_n)=U$, $\pi$ is an open map.

    Now, let us define 
    $$W_n:=\Span(\{(e_m, f_m^n): m\le n\;\textrm{or}\; m>2n\}\cup\{(e_{n+k}, f_{n+k}^k): 1\le k\le n\}).$$
    It is clear by construction that each $W_n$, such as $X_n$, is dense in and isomorphic to $\mathbb R^\infty$. It follows that $W_n$ has no proper dense subspace. Additionally, for $n\in\mathbb N$, we have $W_n\cap W_m=\{0\}$ for $n<m\le 2n$ and $W_n\cap W_m\ne\{0\}$ for $m>2n$, i.e., a family $\mathcal W\subset \{W_n: n\in\mathbb N\}$ of subspaces with pairwise trivial intersection can be arbitrarily large, but never infinite.

    Let $M:=\bigcup_{n=1}^\infty W_n$ and $Y$ be a linear subspace in $M$. Then, since $M$ is a countable union of linear subspaces, $Y\subset W_n$ for some $n\in\mathbb N$, hence, if $Y$ is dense in $X$, then $Y=W_n$ for some $n\in\mathbb N$. This completes the proof.

\end{example}

\begin{remark}\label{no equivalence}
    Note that, because $M:=\mathbb R^\infty\times\{0\}$ is a linear subspace, $M$ is also pointwise dense-lineable and $(\gamma,\aleph_0)$-dense-lineable for all $\gamma<\aleph_0$. Then, for general TVSs and $\alpha\ge\aleph_0$, $\alpha>\gamma$, there is no equivalence between
    \begin{itemize}
        \item $\alpha$-dense-lineability and infinite $\alpha$-dense-lineability;
        \item pointwise $\alpha$-dense-lineability and infinite pointwise $\alpha$-dense-lineability;
        \item $(\gamma,\alpha)$-dense-lineability and infinite $(\gamma,\alpha)$-dense-lineability.
    \end{itemize}
\end{remark}

The last goal in that section is to prove that the class of TVSs where $\alpha$-infinite $\alpha$-dense-lineability and $\alpha$-dense-lineability are equivalent (and the analogous statements for $(\alpha,\beta)$-dense-lineability and pointwise $\alpha$-dense-lineability) includes all infinite-dimensional first-countable TVSs. The next theorem is a somewhat folkloric result whose demonstration we include for the sake of completeness.

\begin{theorem}\label{pseudometrizable}
 Let $X$ be an infinite-dimensional first-countable topological vector space. Then $w(X)\leq \dim(X)$.   
\end{theorem}
\begin{proof}
By \cite[Theorem 4.8.3]{beckenstein2010}, every first-countable TVS is pseudometrizable.

We build a basis for the topology with cardinality $\dim(X)$. Let $\mathcal B$ be a basis of $X$ and $S\subset\mathbb K$ be a dense countable set. Let $T:=\{\lambda_1e_1+\dots+\lambda_ne_n:\lambda_1,\dots,\lambda_n\in S,~e_1,\dots,e_n\in\mathcal B\}$. Since $S$ is countable, we have $|T|=\dim(X)$.

Now, for each $x\in T$, define $C_x:=\{B(x, r): r\in\mathbb Q,\;r>0\}$, where $B(x, r)$ denotes the open ball of center $x$ and radius $r$ with respect to the pseudometric $d$ that generates the topology of $X$. Of course, each $C_x$ is countable, hence, $C:=\bigcup_{x\in T} C_x$ has cardinality $\dim(X)$.

We claim that $C$ is a topological basis for $X$. In fact, let $x:=\mu_1e_1+\dots+\mu_ne_n$, $\mu_1,\dots,\mu_n\in\mathbb K$, $e_1,\dots,e_n\in\mathcal B$. Because $S$ is dense in $\mathbb K$, there are sequences $(\lambda_m^i)_{m\in\mathbb N}$ in $S$ such that $\lim_{m\to\infty} \lambda_m^i = \mu_i$ for each $i\in\{1,\dots,n\}$. Due to the continuity of sum and scalar multiplication, we have $\lim_{m\to\infty}\lambda_m^1e_1+\dots+\lambda_m^ne_n = x$; therefore, $T$ is dense in $X$. Now, let $U$ be an open set. Since $X$ is pseudometrizable, there is $\varepsilon>0$ such that $B(x,\varepsilon)\subset U$ and, since $T$ is dense in $X$, there is $y\in B(x,\varepsilon)\cap T$. Let $\delta\in\mathbb Q$ be such that $\delta<\varepsilon-d(x, y)$. If $z\in B(y, \delta)\in C$, then $d(y, z)<\delta$ and $d(x, z)\le d(x, y)+d(y, z)<\varepsilon-\delta+\delta = \varepsilon$; thus, $B(y, \delta)\subset B(x, \varepsilon)\subset U$, and this completes our proof.

\end{proof}

\begin{corollary}\label{equivalence for pseudometrizable}
    Let $X$ be an infinite-dimensional first-countable topological vector space and $\alpha\geq\aleph_0$. Then:
   \begin{enumerate}[label=(\roman*)]
        \item $M$ is $\alpha$-dense-lineable if, and only if, $M$ is $\alpha$-infinitely $\alpha$-dense-lineable.
        \item $M$ is pointwise $\alpha$-dense-lineable if, and only if, $M$ is $\alpha$-infinitely pointwise $\alpha$-dense-lineable.
        \item $M$ is $(\gamma,\alpha)$-dense-lineable if, and only if, $M$ is $\alpha$-infinitely $(\gamma,\alpha)$-dense-lineable.
    \end{enumerate}
\end{corollary}    
\begin{proof}
In the proofs of Theorem \ref{equivalences1} and Theorem \ref{equivalences2}, we use only that $w(Y)\leq\alpha$, Lemma \ref{X-Y dense lineable} and Theorem \ref{many dense subspaces}. Since $Y$ is first-countable and $\alpha$-dimensional, we have $w(Y)\leq\alpha$. Therefore, the proof follows the same steps used in Theorem \ref{equivalences1} and Theorem \ref{equivalences2}.
\end{proof}

We end this section by proposing a question.

\begin{open}
    From Example \ref{no dense subspace} and Example \ref{d<w with dense subspaces}, it is clear that there is nothing we can say about the existence of infinitely many dense subspaces when $w(X)>\dim(X)$. Is there a nice characterization of all TVSs that admit infinitely many linearly independent dense subspaces in terms of their topology?
\end{open}

\section{Infinite dense-algebrability}\label{alg}
    Throughout this section, we call ``algebra'' an associative linear algebra over $\mathbb K$ and ``topological algebra'' an algebra endowed with a topology for which its sum, product and scalar multiplication are continuous.
    For more details on the subject, we recommend \cite{beckenstein1977topological}. We also denote by $\mathbb P_n$ the set of polynomials in $n$ variables with no constant term.

    If $X$ is an algebra and $S$ is a subset of $X$, we denote by $\G{S}$ the algebra generated by $S$, that is, the smaller subalgebra of $X$ containing $S$. We say $S$ is a generator set for $X$ if $X=\G{S}$. An $\alpha$-generated algebra is an algebra such that $\alpha$ is the smallest cardinality for a generator set.

    If $X$ is commutative, the set $\G{S}$ has the form 
    \begin{equation*}
       \G{S}=\{P(x_1,\dots,x_n): n\in\mathbb N,\;P\in\mathbb P_n,\;x_1,\dots,x_n\in S\}. 
    \end{equation*}

    We say $X$ is an $\alpha$-generated free algebra if there is a subset $S$ of $X$ of cardinality $\alpha$ such that any function from $S$ to an algebra $Y$ can be uniquely extended to a homomorphism from $X$ to $Y$. In this case, we say that $S$ is a set of free generators (SFG) of $X$. In general, for a subset $S$ of $X$, we say $S$ is an SFG if $S$ is an SFG for $\G{S}$. If $X$ is a commutative algebra, $S$ being an SFG for $X$ is equivalent to $X=\G{S}$ and, for all $n\in\mathbb N$, all non-zero polynomial $P\in\mathbb P_n$ and all $x_1,\dots,x_n\in S$ pairwise distinct, having $P(x_1,\dots,x_n)\ne 0$. Equivalently, the set
    $\{x_1^{k_1}\dots x_n^{k_n} : n\in\mathbb N,\;x_1,\dots,x_n\in S,\;k_1,\dots,k_n\in\mathbb N\}$
    is a basis for the vector space $X$.

     We recall the definitions of (dense-) algebrability and strongly (dense-) algebrability.

    If $X$ is an algebra (resp. topological algebra) and $M$ is a subset of $X$, we say $M$ is $\alpha$-algebrable (resp. $\alpha$-dense-algebrable) if there is an $\alpha$-generated subalgebra (resp. a dense $\alpha$-generated subalgebra) of $X$ contained in $M\cup\{0\}$. If $X$ is commutative, we say that $M$ is strongly $\alpha$-algebrable (resp. strongly $\alpha$-dense-algebrable) when $M\cup\{0\}$ contains an $\alpha$-generated free subalgebra (resp. a dense $\alpha$-generated free subalgebra) of $X$.

    We now introduce our definitions related to infinite algebrability.
    
    \begin{definition}\label{def infinite algebrability}
        Given an algebra $X$, a subset $M$ of $X$ and cardinal numbers $\alpha,\beta$ with $\alpha\geq\aleph_0$, we say that $M$ is $\alpha$-infinitely $\beta$-algebrable when there is a family $\{Y_\kappa\}_{\kappa<\alpha}$ of $\beta$-generated subalgebras of $X$ with $Y_\kappa\subset M\cup\{0\}$ for each $\kappa<\alpha$ and $Y_{\kappa_1}\cap Y_{\kappa_2}=\{0\}$ when $\kappa_1\ne\kappa_2$. Furthermore, when $X$ is a topological algebra, we say that $M$ is
        \begin{itemize}
            \item $\alpha$-infinitely $\beta$-dense-algebrable when we can take each $Y_\kappa$ to be a dense subalgebra.

            \item $\alpha$-infinitely strongly $\beta$-algebrable when $X$ is commutative and we can take each $Y_\kappa$ to be a free subalgebra.

            \item $\alpha$-infinitely strongly $\beta$-dense-algebrable when $X$ is commutative and we can take each $Y_\kappa$ to be a dense free subalgebra.
        \end{itemize}
    \end{definition}

    We prove a result analogous to Theorem \ref{many dense subspaces} for infinite strongly dense-algebrability (Theorem \ref{many dense subalgebras}). For that, we will need a version of Lemma \ref{X-Y dense lineable} for free algebras (Theorem \ref{X-Y algebrable}).
    Then, we prove that $\alpha$-infinitely strongly $\alpha$-dense-algebrability is equivalent to strongly $\alpha$-dense-algebrability provided that $w(X)\le\alpha$.
    To finish the session, we provide an example showing that dense-algebrability is not equivalent to infinite $\alpha$-dense-algebrability in general and raise some questions about infinite dense-algebrability.

\begin{lemma}\label{Expanding free generator}
    Let $X$ be a commutative free algebra with a set of free generators $A$ and $S\subset A$. Let $R\subset\G{S}$ be a set of free generators and $x\in X\setminus\G{S}$. Then $R\cup\{x\}$ is a set of free generators.
\end{lemma}
\begin{proof}
    Let $r_1,\dots,r_n\in R$ be pairwise distinct and $P\in\mathbb P_{n+1}$ be a non-zero polynomial. We can write
    $$P(X_1,\dots,X_n,X) = F_0(X_1,\dots,X_n)+\sum_{i=1}^N \alpha_iX^i+F_i(X_1,\dots,X_n)X^i$$
    where $\alpha_1,\dots,\alpha_N\in\mathbb K$, $F_0,\dots,F_N\in\mathbb P_n$ and $F_N$ or $\alpha_N$ is non-zero. We aim to show that $P(\mathrm{r},x)\ne 0$, $\mathrm r := (r_1,\dots,r_n)$, so we assume $N>0$ since this is evident otherwise.
    
    Because $x\not\in\G{S}$, we can find $s_1,\dots,s_m,a\in A$ pairwise distinct such that $r_1,\dots, r_n\in\G{s_1,\dots,s_m}$ and $x\in\G{s_1,\dots,s_m,a}\setminus\G{s_1,\dots,s_m}$. Since $F_i(\mathrm r)\in\G{s_1,\dots,s_m}$ for all $i\in\{0,1,\dots,N\}$, there are $J_0,\dots,J_N\in\mathbb P_m$ such that $F_i(\mathrm r) = J_i(\mathrm s)$, $\mathrm s:=(s_1,\dots,s_m)$, $0\le i\le N$. Let $Q\in\mathbb P_{m+1}$ be such that $x = Q(\mathrm s,a)$. Of course, $Q$ is non-zero, so we can write
    $$Q(X_1,\dots,X_m,X) = G_0(X_1,\dots,X_m)+ \sum_{i=1}^M \beta_iX^i+G_i(X_1,\dots,X_m)X^i$$
    where $M>0$, $\beta_1,\dots,\beta_M\in\mathbb K$, $G_0,\dots,G_M\in\mathbb P_m$ and $G_M$ or $\beta_M$ is non-zero.
    Then
    $$\begin{aligned}
        P(\mathrm r,x) &= J_0(\mathrm s)+\sum_{i=1}^N \alpha_ix^i+J_i(\mathrm s)x^i\\
        &= J_0(\mathrm s)+\sum_{i=1}^N \alpha_iQ(\mathrm s,a)^i+J_i(\mathrm s) Q(\mathrm s,a)^i\\
        &= H_0(\mathrm s)+\sum_{i=1}^{MN}\gamma_ia^i+H_i(\mathrm s) a^i
    \end{aligned}$$
    where $H_i\in \mathbb P_m$ for all $i=0,1,\dots,MN$ and $\gamma_{MN}a^{MN}+H_{MN}(\mathrm s)a^{MN}$ is equal to
    $$\alpha_N(\beta_Ma^M+G_M(\mathrm s)a^M)^N+J_N(\mathrm s)(\beta_Ma^M+G_M(\mathrm s)a^M)^N,$$ 
    which is non-zero since $\alpha_N$ or $F_N$ (therefore $J_N$) and $\beta_M$ or $G_M$ are non-zero and $\{s_1,\dots,s_m,a\}$ is an SFG. Furthermore, for the same reasons,
    $$\gamma_{MN}a^{MN}+H_{MN}(\mathrm s)a^{MN}\ne -\left(H_0(\mathrm s)+\sum_{i=1}^{MN-1} \gamma_i a^i+H_i(\mathrm s) a^i\right)$$
    i.e., $P(\mathrm r,x)\ne 0$.
\end{proof}

\begin{theorem}\label{X-Y algebrable}
    Let $X$ be a commutative topological free algebra with a set of free generators $A$. Let $B, C$ be a partition of $A$ and $\alpha=|C|$ be infinite. If $\alpha\ge w(X)$ and $Y=\G{B}$, then $X\setminus Y$ is strongly $\alpha$-dense-algebrable.
\end{theorem}
\begin{proof}
    Given $x\in X$, we define $A(x):=\bigcap\{S\subset A: x\in\G{S}\}$. Note that $A(x)\subset A$ and $x\in\G{A(x)}$. Additionally, $A(x)$ is finite; therefore, given $S\subset X$ infinite, $A(S):=\bigcup_{x\in S} A(x)$ has the same cardinality as $S$.

    Let $\{U_\kappa\}_{\kappa<\alpha}$ be a topological basis for $X$ (we do not require index uniqueness). Assume that every $U_\kappa$ is non-empty. We build by transfinite induction elements $\{w_\kappa\}_{\kappa<\alpha}$ such that $w_\kappa\in U_\kappa$ and $B\cup \{w_\lambda\}_{\lambda<\kappa}$ is an SFG.

    As $Y\subsetneq X$, we have $\mathrm{int}(Y)=\varnothing$ and there is $w_0\in B_0\setminus Y$. $B\cup\{w_0\}$ is an SFG according to Lemma \ref{Expanding free generator}.
    
    Let $\kappa<\alpha$ and suppose, by transfinite induction hypothesis, that $\{w_\lambda\}_{\lambda<\kappa}$ is already defined. Let $Y_\kappa:=\G{B\cup A(\{w_\lambda\}_{\lambda<\kappa})}$. Notice $B\cup A(\{w_\lambda\}_{\lambda<\kappa})\subsetneq A$ since $|A(\{w_\lambda\}_{\lambda<\kappa})| = \kappa<\alpha=|A\setminus B|$, so $Y_\kappa\subsetneq X$ and $\mathrm{int}(Y_\kappa)=\varnothing$, hence, there is $w_\kappa\in B_\kappa\setminus Y_\kappa$. Again, since $B\cup\{w_\lambda\}_{\lambda<\kappa}\subset Y_\kappa$, by Lemma \ref{Expanding free generator}, $B\cup\{w_\lambda\}_{\lambda\le\kappa}$ is an SFG.
    
    This shows the existence of the family $\{x_\kappa\}_{\kappa<\alpha}$. The subset $\{x_\kappa\}_{\kappa<\alpha}$ is dense in $X$; therefore, $Z:=\G{\{x_\kappa\}_{\kappa<\beta}}$ is dense in $X$. Additionally, $B\cup\{x_\kappa\}_{\kappa<\alpha}$ is an SFG, so $Z\subset (X\setminus Y)\cup\{0\}$ and $Z$ is an $\alpha$-generated free subalgebra. Thus, we conclude that $X\setminus Y$ is strongly $\alpha$-dense-algebrable.

\end{proof}

\begin{remark}\label{algebrability criterion}
    The previous result seems to be of independent interest since it is a variation of the classical problem concerning the lineability/spaceability of complements of vector subspaces (see \cite{BERNALGONZALEZ20143997,araujo2024general}). We believe that more research can be carried out on general algebrability criteria for sets of the form $X\setminus Y$ where $Y$ is a subalgebra of an algebra $X$.
\end{remark}

Given a family of SFGs $\{F_\kappa\}_{\kappa<\alpha}$, we say that $\{\G{F_\kappa}\}_{\kappa<\alpha}$ is an algebraically independent family if $\bigcup_{\kappa<\alpha}F_\kappa$ is an SFG and $F_{\kappa_1}\cap F_{\kappa_2}=\varnothing$ for each $\kappa_1\neq \kappa_2$.

\begin{theorem}\label{many dense subalgebras}
    Let $X$ be a commutative topological $\alpha$-generated free algebra such that $w(X)\leq\alpha$ and $\alpha$ is infinite. Then, there is a family $\{Y_\kappa\}_{\kappa<\alpha}$ such that:
    \begin{enumerate}
        \item $Y_\kappa$ is an $\alpha$-generated free dense subalgebra of $X$ for all $\kappa<\alpha$;
        \item $\{Y_\kappa\}_{\kappa<\alpha}$ is algebraically independent (in particular, $Y_{\kappa_1}\cap Y_{\kappa_2}=\{0\}$ when $\kappa_1,\kappa_2<\alpha$ and $\kappa_1\neq\kappa_2$);
    \end{enumerate}
\end{theorem}
\begin{proof}
    Let $A$ be an SFG of $X$, $\{A_\kappa\}_{\kappa<\alpha}$ be a partition of $A$ with $|A_\kappa|=\alpha$ and $X_\kappa:=\G{A_\kappa}$ for each $\kappa<\alpha$. Then $\{X_\kappa\}_{\kappa<\alpha}$ is an algebraically independent family. We build by transfinite induction a family $\{Y_\kappa\}_{\kappa}$ of $\alpha$-generated dense free subalgebras of $X$ such that for each $\kappa<\alpha$ the family $\{X_\mu\}_{\mu>\kappa}\cup\{Y_\mu\}_{\mu\le\kappa}$ is algebraically independent.

    Let $\kappa<\alpha$ and suppose, by transfinite induction hypothesis, that $\{Y_\lambda\}_{\lambda<\kappa}$ is already defined such that for each $\lambda<\kappa$, the family $\{X_\mu\}_{\mu>\lambda}\cup\{Y_\mu\}_{\mu\le\lambda}$ is algebraically independent. Then, $\{X_\mu\}_{\mu\ge \kappa}\cup\{Y_\mu\}_{\mu<\kappa}$ is algebraically independent. In fact, let $F_\lambda$ be an SFG of $Y_\lambda$ for each $\lambda<\kappa$ and $x_1,x_2,\dots,x_n,y_1,\dots,y_m$ be elements of $\left(\bigcup_{\mu\geq\kappa}A_\mu\right)\cup\left(\bigcup_{\mu<\kappa}F_\mu\right)$ where $x_i\in F_{\zeta_i}$ and $y_j\in A_{\eta_j}$ with $\zeta_i<\kappa\leq\eta_j$, $i\in\{1,\dots,n\}$, $j\in\{1,\dots,m\}$. Let $\lambda$ be so that $\zeta_i\leq\lambda<\kappa\le\eta_j$ for all $i\in\{1,\dots,n\}$ and $j\in\{1,\dots,m\}$. Then $\{X_\mu\}_{\mu>\lambda}\cup\{Y_\mu\}_{\mu\le\lambda}$ is algebraically independent, so $P(x_1,x_2,\dots,x_m,y_1,y_2,\dots,y_n)\ne 0$ for every non-zero polynomial in $n+m$ variables, as we wanted.

    Therefore, $Z:=\G{\{A_\mu\}_{\mu\geq \kappa}\cup\{F_\mu\}_{\mu<\kappa}}$ satisfies the hypothesis for Theorem \ref{X-Y algebrable} with $B:=A_\kappa$ and $C:=\left(\bigcup_{\mu>\kappa}A_\mu\right)\cup\left(\bigcup_{\mu<\kappa}F_\mu\right)$; hence, there is an SFG $F_\kappa$ such that $|F_\kappa|=\alpha$, $Y_\kappa:=\G{F_\kappa}$ is a dense subspace of $Z$ and $\left(\bigcup_{\mu>\kappa}A_\mu\right)\cup\left(\bigcup_{\mu\leq \kappa}F_\mu\right)$ is an SFG. Of course, $Z$ is dense in $X$, thus, so is $Y_\kappa$. This shows the existence of the family $\{Y_\kappa\}_{\kappa<\alpha}$.
\end{proof}

\begin{theorem}\label{criterion equivalence infinite algebrability}
    Let $X$ be a commutative topological algebra, $M$ a subset of $X$, $\alpha\geq\aleph_0$ a cardinal and $\alpha\geq w(X)$. Then, $M$ is $\alpha$-infinitely strongly $\alpha$-dense-algebrable if, and only if, $M$ is strongly $\alpha$-dense-algebrable.
\end{theorem}
\begin{proof}
    Let $Y\subset M\cup\{0\}$ be a dense $\alpha$-generated free subalgebra of $X$. Since $w(Y)\leq w(X)\leq \alpha$, according to Theorem \ref{many dense subalgebras}, there is an algebraically independent family $\{Y_\kappa\}_{\kappa<\alpha}$ of dense subalgebras of $Y$. Because $Y$ is dense in $X$, then $Y_\kappa$ is dense in $X$ for each $\kappa<\alpha$. Therefore, $M$ is $\alpha$-infinitely strongly $\alpha$-dense-algebrable.
\end{proof}

\begin{theorem}\label{equivalence inf alg firs-countable}
    Let $X$ be a commutative topological $\alpha$-generated algebra and $M$ a subset of $X$. If $X$ is first-countable and $\alpha\geq\aleph_0$, then $M$ is $\alpha$-infinitely strongly $\alpha$-dense-algebrable if, and only if, $M$ is strongly $\alpha$-dense-algebrable.
\end{theorem}
\begin{proof}
    Analogous to Corollary \ref{equivalence for pseudometrizable}.
\end{proof}

\begin{example}\label{no equivalence inf alg}
    There are dense-algebrable subsets of topological algebras that are not infinitely dense-algebrable.
        
    Let $\mathbb R^\infty$ be as in Example \ref{no dense subspace}. We define over $\mathbb R^\infty$ the product $(x_n)\cdot(y_n):=(x_ny_n)$. It is clear the set $B:=\{\mathbb R^\infty\cap\left(\prod_{n=1}^\infty(-r_n,r_n)\right): 0<r_n\le 1 \;\forall n\in\mathbb N\}$ is a local basis at $0$. Since $\mathbb R^\infty$ is a locally convex TVS (\cite[Proposition 5.4]{robertson1973}) and $B$ is a local basis at $0$ containing only idempotents sets ($U$ is said idempotent when $U\cdot U\subset U$), $\mathbb R^\infty$ is a topological algebra by \cite[4.3.1]{beckenstein1977topological}.

    Therefore, $\mathbb R^\infty$ is an example of an infinitely generated topological algebra without any proper subalgebra (provided that subalgebras are linear subspaces in particular). Hence, $\mathbb R^\infty$ is dense-algebrable, but not infinitely dense-algebrable. 
\end{example}

We end that work with two more questions.

\begin{open}
    Let $X$ be a topological algebra with $w(X)\leq\alpha$ where $\alpha$ is an infinite cardinal and $M$ a subset of $X$. If $M$ is $\alpha$-dense-algebrable, then is $M$ $\alpha$-infinitely $\alpha$-dense-algebrable?
\end{open}

\begin{open}
    Is there a subset of a commutative topological algebra that is strongly $\alpha$-dense-algebrable, but not infinitely $\alpha$-dense-algebrable?
\end{open}

\bibliographystyle{ieeetr}

\bibliography{bibliography.bib}

\Addresses

\end{document}